\documentclass[12pt,a4paper]{article}
\usepackage[pdftex]{graphicx}
\usepackage{float}
\pdfoutput=1 
\usepackage{booktabs,caption,fixltx2e}
\usepackage[utf8]{inputenc}
\usepackage{hyperref}
\usepackage{graphicx}
\graphicspath{ {./images/} }
\usepackage{amssymb}
\usepackage{amsmath}
\usepackage[mathscr]{euscript}
\usepackage{xcolor}
\usepackage{csquotes}
\usepackage[shortlabels]{enumitem}
\pagecolor{white}
\usepackage[all]{xy}
\usepackage{authblk}
\usepackage{theoremref}
\usepackage{tabularx}

\usepackage{textcomp}
\usepackage{lipsum}
\newcommand\blfootnote[1]{%
  \begingroup
  \renewcommand\thefootnote{}\footnote{#1}%
  \addtocounter{footnote}{-1}%
  \endgroup
}
\usepackage{array}
\usepackage{url}
\usepackage{mathtools}
\usepackage{csquotes}
\usepackage{hyperref}
\usepackage{amsmath,amscd}
\usepackage{authblk}
\usepackage{tikz}
\usepackage{amsthm}
\usepackage{tikz-cd}
\tikzcdset{scale cd/.style={every label/.append style={scale=#1},
    cells={nodes={scale=#1}}}}
\newsavebox{\tempbox}
\theoremstyle{definition}
\newtheorem*{definition}{Definition}
\newtheorem{theorem}{Theorem}[section]
\newtheorem{theo}{Theorem}
\newtheorem{lemma}[theorem]{Lemma}
\newtheorem{prop}[theorem]{Proposition}

\newtheorem{remark}[theorem]{Remark}
\newtheorem{remarkdf}[theorem]{Remark-Definition}

\title{\normalsize \textbf{HAEFLIGER'S APPROACH FOR SPHERICAL KNOTS MODULO IMMERSIONS}}
\date{}
 \author{\normalsize Neeti Gauniyal}
\begin{document}

\newcommand{\Address}{{
  \bigskip
  \small

 \textsc{Department of Mathematics, Kansas State University,
    128 Cardwell Hall, Manhattan, Kansas 66506, USA}\par\nopagebreak
  \textit{Email address}: \texttt{neetigl@ksu.edu}

}}

\maketitle 
\vspace{-0.7cm}

\begin{abstract}
We show that for the spaces of spherical embeddings modulo immersions $\overline{Emb}(S^n,S^{n+q})$ and long embeddings modulo immersions $\overline{Emb}_{\partial}(D^n,D^{n+q})$, the set of connected components is isomorphic to $\pi_{n+1}(SG,SG_q)$ for $q\geq 3$. As a consequence, we show that all the terms of the long exact sequence of the triad $(SG;SO,SG_q)$ have a geometric meaning relating to spherical embeddings and immersions.

\end{abstract}

\section*{Introduction}\blfootnote{The author has benefited from a guest position at the Max Planck Institute for Mathematics in Bonn. Her travel was also partially supported by her advisor V. Turchin's Simon Foundation grant, award ID: 519474.}
Let $Emb(S^n,S^{n+q})$ be the space of smooth embeddings $S^n\hookrightarrow S^{n+q}$ and $Imm(S^n,S^{n+q})$ be the space of smooth immersions $S^n\looparrowright S^{n+q}$. We define the space of \textit{spherical embeddings modulo immersions} $\overline{Emb}(S^n,S^{n+q})$ as the homotopy fiber of $Emb(S^n,S^{n+q})\hookrightarrow Imm(S^n,S^{n+q})$ over the trivial inclusion $id:S^n\subset S^{n+q}$. An element in this space is represented by a pair $(f,\alpha)$, where $f:S^n\hookrightarrow S^{n+q}$ is a smooth embedding together with a regular homotopy $\alpha:[0,1]\rightarrow Imm(S^n,S^{n+q})$, between $f$ and the trivial inclusion $S^n\subset S^{n+q}$. Moreover, $Emb_{\partial}(D^n, D^{n+q})$ denotes the space of disk embeddings $D^n\hookrightarrow D^{n+q}$ with the fixed behavior near the boundary, and we similarly define $\overline{Emb}_{\partial}(D^n,D^{n+q})$ as the space of disk embeddings modulo immersions. For framed spherical/disk embeddings, we consider the spaces 
\sloppy
$Emb^{fr}(S^n,S^{n+q}),\overline{Emb}^{fr}(S^n,S^{n+q}),Emb_{\partial}^{fr}(D^n,D^{n+q})$ and $\overline{Emb}_{\partial}^{fr}(D^n,D^{n+q})$ in the same manner. Throughout the paper for spherical embeddings we assume that the framing respects the natural orientation: if one takes the orientation of $S^n$ and completes it with the orientation of the normal bundle induced by the framing, one obtains the standard orientation of the ambient sphere $S^{n+q}$. For disk embeddings the framing is standard near the boundary. The spaces of embeddings modulo immersions recently attracted a lot of attention \cite{vt,vt1,bw,vt2,sak,vt3}. The main objective of this paper is to revise Haefliger's work \cite{hae}, and apply it to compute $\pi_0$ of these spaces.\\

In \cite[Theorems 3.4 and 5.7]{hae}, Haefliger has shown that for $q\geq 3$, the group of isotopy classes of (framed) spherical embeddings of $S^n$ in $S^{n+q}$ can be represented in terms of the homotopy group of a triad i.e., \sloppy $C_n^q:=\pi_{0}Emb(S^n,S^{n+q})=\pi_{n+1}(SG;SO,SG_q)$ and $FC_n^q:=\pi_{0}Emb^{fr}(S^n,S^{n+q})=\tilde{\pi}_{n+1}(SG;SO,SG_q)$. We recall these homotopy groups and isomorphisms later in \S \ref{sec1}.\\

\noindent
\textbf{Main results.} Let $\overline{FC}{}_n^q$ denote the group of isotopy classes of \enquote{framed disked embeddings}, which we discuss in more detail in \S \ref{sec2}.

\begin{theo}\thlabel{mth1}

\textit{For} $q\geq 3$,\\
$\pi_{0}\overline{Emb}(S^n,S^{n+q})=\pi_{0}\overline{Emb}_{\partial}(D^n,D^{n+q})=\pi_{0}\overline{Emb}(S^n,\mathbb{R}^{n+q})
=\pi_{0}\overline{Emb}^{fr}(S^n,S^{n+q})=\pi_{0}\overline{Emb}_{\partial}^{fr}(D^n,D^{n+q})=\pi_{0}\overline{Emb}^{fr}(S^n,\mathbb{R}^{n+q})
=\overline{FC}{}_n^q=\pi_{n+1}(SG,SG_q).$
\end{theo}
\noindent
The result is an immediate corollary of Theorems \ref{mt} and \ref{cor}.\\

Alternatively, this result can be obtained using smoothing theory, as a consequence of \cite[\S 6]{hae1} and \cite[Theorem~1.1]{sak}. There is even a stronger result: 
\begin{equation*}\pi_i\overline{Emb}_{\partial}(D^n,D^{n+q})= \pi_{n+i+1}(SG_{n+q},SG_q) \text{   for  } i\leq2q-5,
\end{equation*} which follows from the work of Lashof \cite{las1}, Millett \cite[Theorem~2.3]{mill} and Sakai \cite[Theorem 1.1 and Remark 2.3]{sak}. Moreover, for $ i\leq q-3$, $ \pi_{n+i+1}(SG_{n+q},SG_q)=\pi_{n+i+1}(SG,SG_q)$, see Lemma~\ref{le}. However, our goal is to review Haefliger's construction and give a geometric meaning to $ \pi_{n+1}(SG,SG_q)$ i.e., $ \pi_{n+1}(SG,SG_q)=\overline{FC}{}_n^q$. \\

Another main result that does not immediately follow from smoothing theory is to geometrically interpret the long exact sequences associated with the triad $(SG;SO,SG_q)$ considered by Haefliger \cite[\S 4.4 and \S 5.9]{hae}. Let $Im_n^q$ and $FIm_n^q$ denote the group of regular homotopy classes of immersions $S^n\looparrowright S^{n+q}$ and framed immersions $S^n\looparrowright S^{n+q}$, respectively. It is natural to ask which (framed) spherical immersions can be realized as (framed) embeddings, or when two (framed) spherical embeddings are equivalent as (framed) immersions. Answers to these questions are encoded by the lower exact sequences in (\ref{*}) and (\ref{**}) of Theorem~\ref{mth2}, in which $\overline{FC}{}_n^q$ naturally fits.

\begin{theo}\thlabel{mth2}

\textit{For $q\geq 3$, the two long exact sequences of the triad $(SG;SO,SG_q)$ are isomorphic to the corresponding geometric long exact sequences:}
\small
 \begin{equation}\label{*}
\xymatrix@C=0.8em{
\ar[r] &   \pi_{n+1}(SG,SG_q) \ar[r] \ar@{=}[d]    &  \pi_{n+1}(SG;SO,SG_q) \ar[r] \ar@{=}[d]     &  \pi_n(SO,SO_q) \ar[r] \ar@{=}[d] & \pi_n(SG,SG_q) \ar[r] \ar@{=}[d]&\\
\ar[r]   &  \overline{FC}{}_n^q \ar[r]& C_n^q  \ar[r]& Im_n^q \ar[r] & \overline{FC}{}_{n-1}^q \ar[r] & 
}
\end{equation}

 \begin{equation}\label{**}
\xymatrix@C=1em{
\ar[r] &   \pi_{n+1}(SG,SG_q) \ar[r] \ar@{=}[d]    &  \tilde{\pi}_{n+1}(SG;SO,SG_q) \ar[r] \ar@{=}[d]     &  \pi_n(SO) \ar[r] \ar@{=}[d] & \pi_n(SG,SG_q) \ar[r] \ar@{=}[d]&\\
\ar[r]   &  \overline{FC}{}_n^q \ar[r]& FC_n^q  \ar[r]& FIm_n^q \ar[r] & \overline{FC}{}_{n-1}^q \ar[r] & 
}
\end{equation}
\normalsize
\end{theo}

Note that the upper sequences in (\ref{*}) and (\ref{**}) are the long exact sequences of the homotopy groups of pairs $(SG/SG_q,SO/SO_q)$ and $(SG/SG_q, SO)$, respectively, see Remarks \ref{remar} and \ref{remark}.\\

The paper is organized as follows: we give a quick review of Haefliger's result \cite{hae} for (framed) spherical embeddings $S^n\hookrightarrow S^{n+q}$ in \S \ref{sec1}. In \S \ref{sec2}, we define the group $\overline{FC}_n^q$ and show that $\overline{FC}{}_n^q= \pi_{n+1}(SG,SG_q)$. We prove Theorems~\ref{mth1}~and~\ref{mth2} in \S \ref{sec3} and \S \ref{sec4}, respectively. We recall some computations and prove a few applications of Theorem~\ref{mth1} in \S \ref{sec5}. Throughout the paper we work in the smooth category and assume $q\geq 3$. \\

\noindent
\textbf{Terminology:}
Let $D^n$ be the standard unit disk in $\mathbb{R}^n$, and $\{e_1,..,e_n\}$ denote the natural basis of $\mathbb{R}^n$. Let $S^n=\partial D^{n+1}$ be the unit sphere such that $S^n=D_{-}^n\cup D_{+}^n$ with $D_{-}^n=\{x\in S^n | x_1\leq 0\}$ and $D_{+}^n=\{x\in S^n | x_1\geq 0\}$. According to Haefliger \cite{hae}, the \textit{suspension} of a map $f:D^n\rightarrow D^n$ is given by the map $S(f):D^{n+1}\rightarrow D^{n+1}$ sending the arc of circle going from $e_{n+1}$, by $x\in D^n$, to $-e_{n+1}$ on the arc of circle from $e_{n+1}$, by $f(x)$, to $-e_{n+1}$. The suspension $S^{n+1}\rightarrow S^{n+1}$ of a map $S^n\rightarrow S^n$ is defined in the same way.\\

Abusing terminology, the suspension of an embedding $S^n\xhookrightarrow{f} S^{n+q}$ is the composition $S^n\xhookrightarrow{f} S^{n+q}\subset S^{n+q+1}$. For the suspension of a framed embedding $S^n\hookrightarrow S^{n+q}$, the framing is completed by adding the standard vector $e_{n+q+2}$ as the last vector. We often say suspension for an iterated suspension defined inductively. For example, when we say $S^n\hookrightarrow S^{n+N}$ is the suspension of a framed embedding $S^n\xhookrightarrow{f} S^{n+q}$ for $N>q$, we mean that it is defined as the composition $S^n\xhookrightarrow{f} S^{n+q}\subset S^{n+N}$ and the framing is obtained by adding vectors $\{e_{n+q+2}, \ldots, e_{n+N+1}\}$ to the initial framing. We define the suspension of a (framed) disk embedding similarly. \\

\noindent
\textbf{Acknowledgment:}
The author would like to thank her advisor Victor Turchin for his time and helpful feedback on innumerable drafts of this paper.

\section{Embeddings of $S^n$ in $S^{n+q}$}\label{sec1}

Haefliger \cite{hae} proved that the group of concordance classes of embeddings of $S^n$ in $S^{n+q}$ is isomorphic to $\pi_{n+1}(SG;SO,SG_q)$ for $q\geq 3$. 
\subsection{The group $C_n^q$}
\begin{center}
$C_n^q:=\{$concordance classes of smooth embeddings $S^n\hookrightarrow S^{n+q}\}.$  
 \end{center}

\begin{theorem} \cite[Theorem 1.2]{hae}.
\textit{Two concordant embeddings of $S^n$ in $S^{n+q}$ are isotopic when $q\geq3$, i.e. $C_n^q=\pi_{0}Emb(S^n,S^{n+q})$, the set of connected components of the space of embeddings of $S^n$ in $S^{n+q}$.}
\end{theorem}

Furthermore, the equality $\pi_{0}Emb(S^n,S^{n+q})= \pi_{0}Emb_{\partial}(D^n, D^{n+q})$ enables $C_n^q$ with an additive multiplication, and the existence of inverses is guaranteed, as we consider concordance classes. Hence, $C_n^q$ is an abelian group.
\begin{lemma}\cite[\S 1]{hae}.\thlabel{slice}
\textit{An embedding $S^n\hookrightarrow S^{n+q}$ is concordant to the trivial one if and only if it is slice, in other words, if it can be extended to an embedding $D^{n+1}\hookrightarrow D^{n+q+1}$.}
\end{lemma}

\subsection{The group $\pi_{n+1}(SG;SO,SG_q)$}\label{triad}
Let $SG_q$ be the space of degree one maps $S^{q-1}\rightarrow S^{q-1}$, $SG=\cup SG_q$ under suspension, and $SO=\cup SO_q$, where $SO_q$ is the special orthogonal group.\\

An element in $\pi_{n+1}(SG;SO,SG_q)$ is represented by a continuous based map $\phi:D^{n+1}\rightarrow SG$ i.e., for $x\in D^{n+1}$, $\phi(x):S^{N-1}\rightarrow S^{N-1}$, for some large $N$, such that $\phi(D_{-}^n)\subset SO_N$ and $\phi(D_{+}^n)\subset SG_q$. Note that the equator $S^{n-1}=\partial D_{-}^n=\partial D_{+}^n$ goes to $SO\cap SG_q=SO_q$, and $\phi(*)=id$ for the base-point $*=e_2\in S^{n-1}$.\footnote{Haefliger in \cite{hae} does not consider the base-point condition, but it is immediate that adding it yields the same homotopy group, since $SO_q=SO\cap SG_q$ is connected.} Abusing notation, we also view $\phi$ as a map $\phi:D^{n+1}\times S^{N-1}\rightarrow S^{N-1}$, and sometimes for $\phi(x)$ we write $\phi_x=\phi(x,-):S^{N-1}\rightarrow S^{N-1}$.\\

Two such maps $\phi:D^{n+1}\times S^{N-1}\rightarrow S^{N-1}$ and $\phi':D^{n+1}\times S^{N'-1}\rightarrow S^{N'-1}$ represent the same element in $\pi_{n+1}(SG;SO,SG_q)$ if there is a homotopy $\phi_t:D^{n+1}\times S^{M-1}\rightarrow S^{M-1}$ for some $M\geq N, N'$ and $t\in [0,1]$, satisfying the above conditions and such that for any $x\in D^{n+1}$, the maps $\phi_0(x,-),\phi_1(x,-):S^{M-1}\rightarrow S^{M-1}$ are suspensions of $\phi(x,-)$ and $\phi'(x,-)$, respectively. The product operation of any two elements in $\pi_{n+1}(SG;SO,SG_q)$ is defined point-wise.

\setbox\tempbox=\hbox{\begin{tikzcd}[scale cd=0.8]
BSO_q \arrow[r] \arrow[d]& BSG_q \arrow[d] \\
BSO \arrow[r]  & BSG    \text{    .}\end{tikzcd}} 
\begin{remark}\label{remar} Recall that the upper long exact sequence in (\ref{*}) is the long exact sequence of the pair $(SG/SG_q,SO/SO_q)$. Indeed, Milgram \cite[\S 1]{mil} interpreted the group $\pi_{n+1}(SG;SO,SG_q)$ as $\pi_n$\Big(hofib$(SO/SO_q\rightarrow SG/SG_q)\simeq$ hofib$(SG_q/SO_q\rightarrow SG/SO)$\Big).\footnote{The spaces are equivalent because they describe the total homotopy fiber of the square \[\box\tempbox\]} One way to see this interpretation is that the group $\pi_{n+1}(SG;SO,SG_q)$ is obviously isomorphic to $\pi_{n}(SO\times_{SG}^{h} SG_q,SO_q)$, where $SO\times_{SG}^{h} SG_q$ is the homotopy pullback of $SO\rightarrow SG\leftarrow SG_q$. 

\end{remark}

\subsection{The isomorphism $\psi:C_n^q\rightarrow \pi_{n+1}(SG;SO,SG_q)$}\label{sub}

Although these two groups look completely different, there is a natural map between them. To see the relation, Haefliger considers representatives in $C_n^q$ to be framed embeddings of $D^{n+1}$ with different boundary conditions on $D^n_{-}\subset S^n=\partial D^{n+1}$ and $D^n_{+}\subset S^n=\partial D^{n+1}$. Framing will be crucial to relate such embeddings to $\pi_{n+1}(SO;SO,SG_q)$ by means of Pontryagin-Thom type construction \cite[\S 3]{hae}.\\

Given an embedding $f:S^n\hookrightarrow S^{n+q}$, we say $f$ is a \textbf{\textit{special embedding}} if $f|_{D_{-}^{n}}=id$ and $f$(int $D_{+}^n)\subset$ int $D_{+}^{n+q}$. 
We can always extend $f:S^n\hookrightarrow S^{n+q}$ to a disk embedding $\bar{f}:D^{n+1}\hookrightarrow D^{n+N+1}$, for some $N$ large enough (in fact $N> n+2$). We refer the obtained pair $(f,\bar{f}):(S^n,D^{n+1})\hookrightarrow (S^{n+q},D^{n+N+1})$ as a \textbf{\textit{disked embedding}}. Any element in $C_n^q$ can be represented by a special disked emedding $(f,\bar{f})$ together with some framing on $\bar{f}$ defined as follows:
\begin{itemize}
\item Fix the base-point $*=e_2\in S^{n-1}=D_{-}^n\cap D_{+}^n$  and endow it with the framing $\{e_{n+2},\ldots,e_{n+q+1}\}$.
\item Extend the framing from $*=e_2$ to $D_{+}^n$ inside $D_{+}^{n+q}$. Since $*\hookrightarrow D^n_{+}$ is a homotopy equivalence, this extension is unique up to homotopy. Take the suspension of this framing in $D^{n+N+1}$ by adding $\{e_{n+q+2},\ldots,e_{n+N+1}\}$ as last vectors.
\item Extend the obtained framing from $D_{+}^n$ to the entire disk $D^{n+1}$ inside~$D^{n+N+1}$. Again this framing is defined uniquely up to homotopy.\\
\end{itemize}

Note that even though the knot $f$ is trivial on $D_{-}^n$, the extended framing can be non-trivial. Moreover, the framing on $\bar{f}|_{D_{-}^n}$ inside $D^{n+N+1}$ might not be a suspension, while the framing on $\bar{f}|_{D^n_{+}}$ inside $D^{n+N+1}$ is the suspension of a framing inside $D^{n+q}_{+}$. We refer this boundary condition on the framing defined on $\bar{f}$ as \textbf{\textit{Type~I}} (in sections \ref{subsec1} and \ref{sec2}, we will also consider framing with Type~II and Type~III boundary conditions). Hence, any embedding $f:S^n\hookrightarrow S^{n+q}$ representing an element in $C^q_n$ can be considered as a \textbf{\textit{special disked embedding $(f,\bar{f}):(S^n,D^{n+1})\hookrightarrow (S^{n+q},D^{n+N+1})$ with Type~I framing}}, i.e., $\bar{f}|_{S^n=\partial D^{n+1}}=f$ is a special knot, and the framing on $\bar{f}$ has boundary condition defined as above.\\

Any two special disked embeddings $(f_0,\bar{f}_{0}):(S^n,D^{n+1})\hookrightarrow (S^{n+q},D^{n+N_{0}+1})$ and $(f_1,\bar{f}_{1}):(S^n,D^{n+1})\hookrightarrow (S^{n+q},D^{n+N_{1}+1})$ with Type~I framing are \textit{concordant} if there exists an embedding $F:D^{n+1}\times [0,1]\hookrightarrow D^{n+N+1}\times [0,1]$ for $N\geq max\{N_{0},N_{1}\}$, such that $F|_{D^{n+1}\times i}=\bar{f}_i$ for $i=0,1$, $F|_{D_{-}^n\times [0,1]}=id$ and $F|_{D_{+}^n\times [0,1]}\subset D_{+}^{n+q}\times [0,1]$. Furthermore, the framing on $F|_{D_{+}^n\times [0,1]}$ and $F|_{D^{n+1}\times i}$, $i=0,1$, is given by suspension of a framing inside $D_{+}^{n+q}\times [0,1]$ and $D^{n+N_{i}+1}\times i$,  respectively. Similarly, we define the isotopy relation to be a level-preserving concordance. All the following groups are isomorphic for $q\geq 3$:
\begin{center}

  $C^q_n=$\{concordance/isotopy classes of embeddings $S^n\hookrightarrow S^{n+q}$\}\\
   $\updownarrow$\\
  
     \{concordance/isotopy classes of special embeddings $S^n\hookrightarrow S^{n+q}$\}\\
           $\updownarrow$\\
     
  \{concordance/isotopy classes of special disked embeddings $(S^n,D^{n+1})\hookrightarrow (S^{n+q},D^{n+N+1})$\}\\
     $\updownarrow$\\
     
  \{concordance/isotopy classes of special disked embeddings $(S^n,D^{n+1})\hookrightarrow (S^{n+q},D^{n+N+1})$ with Type~I framing\}\\
  
  \end{center}
  
\vspace{0.2cm}  

Furthermore, as a consequence of the tubular neighborhood theorem, one can choose a representative $f$ in $C_n^q$ such that $f(S^n)$ is contained in a subspace of $S^{n+q}$ which can be identified with $S^n\times D^q$. Thus, we can consider a special knot to be $f:S^n\hookrightarrow S^n\times D^q$ such that $f|_{D_{-}^n}$ is the natural inclusion $D_{-}^n\hookrightarrow D_{-}^n\times 0$ and $f($int $D_{+}^n)\subset$ int$(D_{+}^n\times D^q)$, together with a disk extension $\bar{f}:D^{n+1}\hookrightarrow D^{n+1}\times D^N$ with a similarly defined framing of Type~I. The homomorphism $\psi: C_n^q\rightarrow \pi_{n+1}(SG;SO,SG_q)$ is then defined as follows.\\

\begin{theorem}\thlabel{th}
\textit{Given an element $\alpha\in C_n^q$ represented by a special disked embedding $(f,\bar{f}):(S^n, D^{n+1})\hookrightarrow (S^n\times D^q, D^{n+1}\times D^N)$ with Type~I framing, a map $\phi:D^{n+1}\times S^{N-1}\rightarrow S^{N-1}$ represents $\psi(\alpha)\in  \pi_{n+1}(SG;SO,SG_q)$ if there exists an extension $\bar{\phi}:D^{n+1}\times D^N\rightarrow D^N$ i.e., $\bar{\phi}|_{D^{n+1}\times S^{N-1}}=\phi$ such that:}
\begin{enumerate}[(i)]

\item \textit{$\bar{\phi}$ is regular on $0\in D^N$ and $\bar{\phi}^{-1}(0)=\bar{f}(D^{n+1})$ as framed submanifolds,}
\item \textit{$\bar{\phi}_x\in SO_N $ for $x\in D_{-}^n$,}
\item \textit{$\bar{\phi}_x$ is the suspension of a map $D^q\rightarrow D^q$ for $x\in D_{+}^n$.}

\end{enumerate}
\textit{The homomorphism $\psi:C_n^q\rightarrow \pi_{n+1}(SG;SO,SG_q)$ is well defined \cite[Theorem~2.3]{hae} and is an isomorphism for $q\geq 3$ \cite[Theorem~3.4]{hae}.} \\
\end{theorem}

In the proof of well-definedness of $\psi$ \cite[Theorem 2.3]{hae}, Haefliger shows the existence of such a map $\bar{\phi}$ as follows. Define $\bar{\phi}_{-}:D^n_{-}\times D^N\rightarrow D^N$ uniquely as a linear map such that $(\bar{\phi}_{-})_x\in SO_N$ for $x\in D^n_{-}$ and $\bar{\phi}^{-1}_{-}(0)=f(D^n_{-})$, as framed submanifolds. Using obstruction theory \cite[Lemma~2.4]{hae} the restriction $\bar{\phi}_{-}|_{S^{n-1}\times D^q}$ can be extended to $\bar{\phi}_{-}|_{D^n_{+}\times D^q}$ with the given framing on $f(D^n_{+})$. Define $\bar{\phi}_{+}:D^n_{+}\times D^N\rightarrow D^N$ to be the $(N-q)$-suspension of $\bar{\phi}_{-}:{D^n_{+}\times D^q}\rightarrow D^q$. By using \cite[Lemma~2.4]{hae} again, we extend $\bar{\phi}_{-}\cup \bar{\phi}_{+}:S^n\times D^N\rightarrow D^N$ to a map $\bar{\phi}:D^{n+1}\times D^N\rightarrow D^N$ verifying (i)-(iii) above. To show $\psi$ is well defined, he uses the same argument invoking \cite[Lemma~2.4]{hae} twice to construct a homotopy between two maps $\phi_0,\phi_1:D^{n+1}\times S^{N-1}\rightarrow S^{N-1}$ corresponding to two concordant embeddings $(f_0,\bar{f_0}),(f_1,\bar{f_1}):(S^n, D^{n+1})\hookrightarrow (S^n\times D^q, D^{n+1}\times D^N)$.\\

To prove the isomorphism \cite[Theorem 3.4]{hae}, Haefliger interprets the group $\pi_{n+1}(SG;SO,SG_q)$ in terms of cobordisms (we refer this as Pontryagin-Thom type construction). An element of $\pi_{n+1}(SG;SO,SG_q)$ represented by a map $\phi: D^{n+1}\times S^{N-1}\rightarrow S^{N-1}$ as in subsection \ref{triad} which is regular on $e_1$, corresponds to a framed $(n+1)$-submanifold $V=\phi^{-1}(e_1)\subset D^{n+1}\times S^{N-1}$ with two parts of boundary:
\begin{itemize}
\item $V\cap (D^n_{-}\times S^{N-1})$ is the graph of some map $g:D^n_{-}\rightarrow S^{N-1}$ with the framing at points $(x,g(x))$ lying inside $x\times S^{N-1}$ and orthonormal. Indeed, for $x\in D^n_{-}$, the map $\phi_x:S^{N-1}\rightarrow S^{N-1}$ is linear and therefore the preimage of $e_1$ is just a point.
\item $V\cap (D^n_{+}\times S^{N-1})$ is the suspension of a framed submanifold in $D^n_{+}\times S^{q-1}$, since for any $x\in D^n_{+}$, the map $\phi_x:S^{N-1}\rightarrow S^{N-1}$ is the suspension of a map $S^{q-1}\rightarrow S^{q-1}$. 
\end{itemize}
Thus, $\pi_{n+1}(SG;SO,SG_q)$ can be described as the group of cobordisms of framed $(n+1)$-manifolds with such boundary conditions.\\

He then considers $\bar{\phi}$ which exists by \cite[Theorem~2.3]{hae}. Note that $\bar{\phi}:D^{n+1}\times D^N\rightarrow D^N$ can always be slightly changed so that $\bar{\phi}^{-1}(\partial D^N)\subset D^{n+1}\times \partial D^N$. The preimage $\bar{\phi}^{-1}(I)\subset D^{n+1}\times D^N$ of the segment $I$ joining $0$ and $e_1$ in $D^N$ is a framed $(n+2)$-manifold $W$ with corners ($\bar{\phi}$ is chosen to be transversal to $I$). In particular, $\partial W$ has the following strata:
\begin{itemize}
\item a free face given by the framed disk $\bar{f}(D^{n+1})=\bar{\phi}^{-1}(0)$ 
\item $\partial W\cap (D^{n+1}\times S^{N-1})=\bar{\phi}^{-1}(e_1)=V$
\item $\partial W\cap (D^n_{-}\times D^N)$ is the radial extension of $V\cap (D^n_{-}\times S^{N-1})$
\item $\partial W\cap (D^n_{+}\times D^N)$ is the $(N-q)$-fold suspension of a framed submanifold in $D^n_{+}\times D^q$.
\end{itemize}

As a result, he restates the homomorphism defined in Theorem~\ref{th} as follows. Given an element $\alpha\in C_n^q$ represented by a special disked embedding $(f,\bar{f}):(S^n, D^{n+1})\hookrightarrow (S^n\times D^q, D^{n+1}\times D^N)$ with Type~I framing, a framed submanifold $V\subset D^{n+1}\times S^{N-1}$ as defined above represents $\psi(\alpha)\in  \pi_{n+1}(SG;SO,SG_q)$ if there exists a framed submanifold $W\subset D^{n+1}\times D^N$ with the boundary strata as given above.

According to \cite[Argument 3.5]{hae}, to show surjectivity he applies surgery to construct $W$ satisfying $\partial W\cap (D^{n+1}\times S^{N-1})=V$ for a given $V$. For injectivity, he shows if $[(f,\bar{f})]$ maps to the trivial element $[V]$ of $\pi_{n+1}(SG;SO,SG_q)$, then the corresponding $W$ can be modified using surgery so that it is embedded in $D^{n+1}\times D^q$. In particular, the free face $\bar{f}(D^{n+1})$ of $W$ is inside $D^{n+1}\times D^q\cong D^{n+q+1}$, and therefore the corresponding $f=\bar{f}|_{\partial D^{n+1}}$ is slice i.e., concordant to the trivial embedding of $S^n$ in $S^{n+q}$, by Lemma~\ref{slice}. 

\subsection{Framed embeddings of $S^n$ in $S^{n+q}$}\label{subsec1}
Let us recall that we always consider framed embeddings with a framing preserving the natural orientation. For $q\geq 3$, Haefliger expressed the group $FC_n^q$ of concordance classes of framed embeddings of $S^n$ in $S^{n+q}$ as $\tilde{\pi}_{n+1}(SG;SO,SG_q)$. An element in $\tilde{\pi}_{n+1}(SG;SO,SG_q)$ is represented by a continuous map $\phi:D^{n+1}\rightarrow SG$ i.e., for $x\in D^{n+1}$, $\phi(x):S^{N-1}\rightarrow S^{N-1}$, for some large $N$, such that $\phi(D_{-}^n)\subset SO$, $\phi(D_{+}^n)\subset SG_q$ and $\phi(\partial D_{-}^n=\partial D_{+}^n)=~id$. Again, abusing notation we also view $\phi$ as a map $\phi:D^{n+1}\times S^{N-1}\rightarrow S^{N-1}$ and sometimes write $\phi_x$ for $\phi(x)$.

\begin{remark}\label{remark}It is easy to see that the group $\tilde{\pi}_{n+1}(SG;SO,SG_q)$ is isomorphic to $\pi_{n}\Big((SO\times_{SG}^{h} SG_q)\simeq$ hofib$(SO\rightarrow SG/SG_q)\Big)$. Moreover, the upper long exact sequence in (\ref{**}) is the long exact sequence of the pair $(SG/SG_q,SO)$.
\end{remark}

\begin{remark}\cite[\S 5.1]{hae}.
Two concordant framed embeddings of $S^n$ in $S^{n+q}$ are isotopic when $q\geq3$ and therefore, $FC_n^q=\pi_{0}Emb^{fr}(S^n,S^{n+q})=\pi_{0}Emb_{\partial}^{fr}(D^n, D^{n+q})$.
 \end{remark}

\begin{lemma}\cite[\S 5]{hae}.\thlabel{slice2}
\textit{A framed embedding $S^n\hookrightarrow S^{n+q}$ is concordant to the trivial one if and only if it is slice i.e., if it can be extended to an embedding $D^{n+1}\hookrightarrow D^{n+q+1}$ along with the framing.}
\end{lemma}
 
\subsection*{The isomorphism $\tilde{\psi}:FC_n^q\rightarrow \tilde{\pi}_{n+1}(SG;SO,SG_q)$}
 
The natural map $\tilde{\psi}$ between the two groups is defined as in the \enquote{non-framed} case. Firstly, an element in $FC_n^q$ can be represented by a special framed knot $f:S^n\hookrightarrow S^{n+q}$ which is the natural inclusion on $D_{-}^n$ with trivial framing $\{e_{n+2},\ldots,e_{n+q+1}\}$, and $f($int $D_{+}^n)\subset$ int$(D_{+}^{n+q})$ with some non-trivial framing. Such a framed knot is assigned a special disked embedding $(f,\bar{f}):(S^n,D^{n+1})\hookrightarrow (S^{n+q},D^{n+N+1})$ along with a framing as follows. We extend $f:S^n\hookrightarrow S^{n+q}$ to a disk embedding $\bar{f}:D^{n+1}\hookrightarrow D^{n+N+1}$ for $N$ large enough. For the framing on $\bar{f}(D^{n+1})$, which is defined uniquely up to homotopy, we first suspend the framing on $D_{+}^n$ inside $D_{+}^{n+q}$ to a framing inside $D^{n+N+1}$ by adding vectors $\{e_{n+q+2},\ldots,e_{n+N+1}\}$. Then we extend the obtained framing to the entire disk $D^{n+1}$ inside $D^{n+N+1}$. Note that the framing on $\bar{f}|_{D_{-}^n}$ may now be non-trivial (and does not have to be a suspension), while the framing on $\bar{f}|_{D_{+}^n}$ is the suspension of the framing on $D^n_{+}$ inside~$D^{n+q}_{+}$. But we still obtain a trivial framing on the equator $S^{n-1}=D_{-}^n\cap D_{+}^n$. Such boundary condition on the framing defined on $\bar{f}$ is referred as \textbf{\textit{Type~II}}. Therefore, a representative in $FC_n^q$ can be considered to be a \textbf{\textit{special disked embedding $(f,\bar{f})$ with Type~II framing}}. For $q\geq 3$, the following groups are isomorphic:

\begin{center}

  $FC^q_n=$\{concordance/isotopy classes of framed embeddings $S^n\hookrightarrow S^{n+q}$\}\\
   $\updownarrow$\\
  
     \{concordance/isotopy classes of special framed embeddings $S^n\hookrightarrow S^{n+q}$\}\\
           $\updownarrow$\\
     
  \{concordance/isotopy classes of special disked embeddings $(S^n,D^{n+1})\hookrightarrow (S^{n+q},D^{n+N+1})$ with Type~II framing\}\\
  
  \end{center}
  
\vspace{0.2cm}

Using the tubular neighborhood theorem, we can transform any special framed knot $f:S^n\hookrightarrow S^{n+q}$ into $f:S^n\hookrightarrow S^n\times D^q$, with a framed disk extension $\bar{f}:D^{n+1}\hookrightarrow D^{n+1}\times D^N$. Thus, an element in $FC^q_n$ can be represented by a pair $(f,\bar{f}):(S^n,D^{n+1})\hookrightarrow (S^n\times D^q, D^{n+1}\times D^N)$ with a framing of Type~II. We define the homomorphism $\tilde{\psi}: FC_n^q\rightarrow \tilde{\pi}_{n+1}(SG;SO,SG_q)$ exactly as in Theorem~\ref{th} by adding to condition ii) that $\bar{\phi}_x=id$ when $x\in S^{n-1}$. For $q\geq 3$, $\tilde{\psi}$ is an isomorphism \cite[Theorem~5.7]{hae}. This result is stated without proof because the argument follows the same lines as in the \enquote{non-framed} case. Note that Lemma~\ref{slice2} is used in the proof of injectivity of $\tilde{\psi}$ in the same way as Lemma~\ref{slice} is necessary for injectivity of $\psi$.

\section{Framed disked embeddings}\label{sec2}

We now define a new group of concordance classes of special disked embeddings with a \textbf{\textit{Type~III}} framing. Namely, this time we require the framing to be trivial along $D_{-}^n$. 
To be precise, we consider special disked embeddings $(f,\bar{f}):(S^n,D^{n+1})\hookrightarrow (S^{n+q},D^{n+N+1})$ where the framing on $\bar{f}$ comes with the following boundary condition: $\bar{f}|_{D_{-}^n}$ has trivial framing, while the framing on $\bar{f}|_{D_{+}^n}$ inside $D^{n+N+1}$ is obtained as the suspension of a framing inside~$D_{+}^{n+q}$. \\
\begin{multline*}
\overline{FC}{}_n^q:=\{ \text{concordance classes of special disked embeddings} \\(f,\bar{f}):(S^n,D^{n+1})\hookrightarrow (S^{n+q},D^{n+N+1}) \text{ with Type~III framing}\}.
\end{multline*}

Note that since the codimension condition $q\geq 3$ is satisfied, concordance and isotopy relations coincide for special disked embeddings with all three boundary restrictions on framing. 

\subsection{The group $\pi_{n+1}(SG,SG_q)$}
An element in $\pi_{n+1}(SG,SG_q)$ is represented by a continuous map $\phi:D^{n+1}\rightarrow SG$ such that $\phi|_{D_{-}^n}= id$ and $\phi(D_{+}^n)\subset SG_q$.\\

This representation is equivalent to the usual definition of a relative homotopy group i.e., $\pi_{n+1}(SG;*,SG_q)=\pi_{n+1}(SG,SG_q)$, since $D_{-}^n$ can be collapsed to get the base-point in the relative group.

\subsection{The isomorphism $\xi:\overline{FC}{}_n^q\rightarrow \pi_{n+1}(SG,SG_q)$}\label{sub3}
Following the same argument as in subsection \ref{sub}, when an element in $\overline{FC}{}_n^q$ is represented by a special disked embedding $(f,\bar{f}):(S^n,D^{n+1})\hookrightarrow (S^n\times D^q,D^{n+1}\times D^N)$ with Type~III framing, there is a natural homomorphism $\xi: \overline{FC}{}_n^q\rightarrow \pi_{n+1}(SG,SG_q)$ defined as in Theorem~\ref{th} by replacing condition ii) with $\bar{\phi}_x=id$ for $x\in D_{-}^n$. By Haefliger's surgery construction \cite[Argument~3.5]{hae} that proves \cite[Theorem~3.4]{hae}, we conclude:
\begin{theorem}\thlabel{mt}
\textit{The homomorphism $\xi:\overline{FC}{}_n^q\rightarrow \pi_{n+1}(SG,SG_q)$ is an isomorphism for $q\geq3$.}
\end{theorem}

The sliceness Lemma~\ref{slice3} is used to prove injectivity of $\xi$, similarly to the cases of $\psi$ and $\tilde{\psi}$. Note that Theorem~\ref{mt} can be deduced from the proof of Theorem~\ref{mth2} given in section \ref{sec4}. In particular, with $\psi$ and $\eta$ as isomorphisms in (\ref{five}), $\xi$ is also an isomorphism by the five lemma.\\

As a review, the following tables point to the main difference among all the groups we discussed in the three cases. In terms of special disked embeddings with different boundary conditions on framing of $\bar{f}$:\\
 \begin{center}
\begin{tabular}{| l | r |}
\hline
$C_n^q$  &   trivial framing at the base-point $*$ (Type I)\\
 \hline
$FC_n^q$ & trivial framing at the equator $S^{n-1}$ (Type II) \\
 \hline
$\overline{FC}{}_n^q$ & trivial framing at $D_{-}^n$ (Type III)\\ \hline

\end{tabular}
    \end{center}
    
\vspace{0.2cm}
    
The corresponding homotopy groups differ as follows:  \\  
  \begin{center}
\begin{tabular}{| l | r |}
\hline

$\pi_{n+1}(SG;SO,SG_q)$  &  $\phi(*)=id$\\
 \hline

 $\tilde{\pi}_{n+1}(SG;SO,SG_q)$  &  $\phi(S^{n-1})=id$\\
 \hline

$\pi_{n+1}(SG,SG_q)$ & $\phi(D_{-}^n)=id$\\
 \hline

\end{tabular}
    \end{center}
    
\vspace{0.2cm}
    
\begin{remarkdf}\label{rkdf}

Consider a disked embedding $(f,\bar{f}):(S^n,D^{n+1})\hookrightarrow (S^{n+q},D^{n+N+1})$ which is not necessarily special, i.e., without a fixed behavior at $D_{-}^n$. Assume both $f$ and $\bar{f}$ are framed embeddings such that framing on $\bar{f}(D^{n+1})$ inside $D^{n+N+1}$ is defined by extending the suspension of the framing of $f(S^n)\subset S^{n+q}$. We call such a pair $(f,\bar{f})$ a \textbf{\textit{framed disked embedding}}. The concordance classes of such embeddings are the same as those of special ones with Type~III framing representing elements in $\overline{FC}{}_{n}^q$. It is because given any framed disked embedding, we can always isotope it, so that near the base-point $*\in \partial D_{-}^n=S^{n-1}$ it is the identity inclusion with the trivial framing. Then we can reparametrize the sphere so that the small neighborhood of $*$ is $D_{-}^n$ and the rest is $D_{+}^n$. As a result, we get a special disked embedding $(f,\bar{f})$ with Type~III framing. Therefore, we can describe $\overline{FC}{}_n^q$ as the group of concordance classes of framed disked embeddings $(f,\bar{f})$. 
 \end{remarkdf}

Thus, all the groups $C_n^q$, $FC_n^q$ and $\overline{FC}{}_n^q$ can be described as groups of concordance classes of \enquote{non-special} embeddings:\\

 \begin{center}
\begin{tabular}{| l | r |}
\hline
$C_n^q$  &  embeddings $S^n\hookrightarrow S^{n+q}$\\
 \hline
$FC_n^q$ & framed embeddings $S^n\hookrightarrow S^{n+q}$  \\
 \hline
\raisebox{-0.8ex}{$\overline{FC}{}_n^q$} & \raisebox{-0.5ex}{framed disked embeddings} \\ &\raisebox{0.2ex} {$(S^n,D^{n+1})\hookrightarrow (S^{n+q},D^{n+N+1})$}\\  \hline

\end{tabular}
    \end{center}

Note that special disked embeddings with framing of Type~I or Type~II are not framed disked embeddings because for latter we require the framing on $\bar{f}$ to be the suspension on entire boundary $\partial D^{n+1}=S^n$, see the definition above.

\subsection{Sliceness}

In this subsection, we study an interesting property of sliceness for framed disked embeddings representing elements in the group $\overline{FC}{}_n^q$. 
\begin{definition}
A framed disked embedding $(f,\alpha): (S^n,D^{n+1})\hookrightarrow (S^{n+q},D^{n+N+1})$ is \textbf{slice} if there exists a framed embedding $H:D^{n+2}\hookrightarrow D^{n+N'+2}$ where $N'\geq N$, such that $ H|_{(\partial_{-}D^{n+2}=D_{-}^{n+1})}=\alpha$ and $H|_{\partial_{+}D^{n+2}}$ is the suspension of a framed embedding inside $D^{n+q+1}$ i.e., $H(\partial_{+}D^{n+2})\subset D^{n+q+1}\subset \partial_{+}D^{n+N'+2}= D^{n+N+1}.$
 \end{definition}
 
The trivial element in $\overline{FC}{}_n^q$ is given by the equivalence class of the trivial framed disked embedding $(id,id):(S^n,D^{n+1})\subset (S^{n+q},D^{n+N+1})$ i.e., the trivial pair with the trivial framing. \\

\begin{lemma}\thlabel{slice3}
\textit{A framed disked embedding $(f,\alpha): (S^n,D^{n+1})\hookrightarrow (S^{n+q},D^{n+N+1})$ representing an element in $\overline{FC}{}_n^q$ is concordant to the trivial element $(id,id):(S^n,D^{n+1})\subset  (S^{n+q},D^{n+N+1})$, if and only if $(f,\alpha)$ is slice.}

\end{lemma}

\begin{proof} Let $F:D^{n+1}\times [0,1]\hookrightarrow D^{n+N'+1}\times [0,1]$, where $N'\geq N$ be a concordance between $(f,\alpha)$ and $(id,id)$. Since at $t=1$ we have a trivial framing, we attach a half disk $\frac{1}{2}D^{n+N'+2}$ along the trivial embedding such that it extends $D^{n+1}$ to the disk $D^{n+2}$, see Figure \ref{fig1}. Since $F$ takes the boundary inside $S^{n+q}\times [0,1]$, therefore attaching this half disk gives the sliceness of the framed knot $S^n\hookrightarrow S^{n+q}$ i.e., a framed extension $D^{n+1}\hookrightarrow D^{n+q+1}$. As a result, we get a framed embedding $H:D^{n+2}\hookrightarrow D^{n+N'+2}$, which on one part of $\partial D^{n+2}$ gives $\alpha$ and on the other, an embedding to $D^{n+q+1}$.
Therefore, $(f,\alpha)$ is slice. \\  

\begin{figure}[htbp]
\centering

\includegraphics[width=0.7\textwidth]{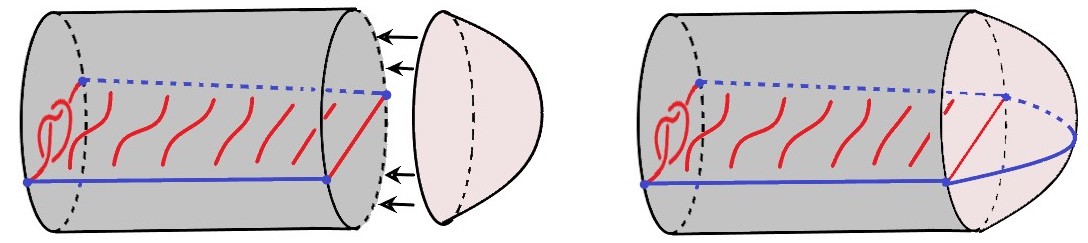}
 
\caption{Attaching a half disk $D^{n+N'+2}$ to $D^{n+N'+1}$ at $t=1$.}
\label{fig1}
\end{figure}
\noindent

\begin{figure}[htbp]
\centering
\includegraphics[width=0.15\textwidth]{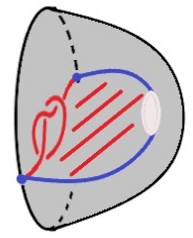}
\caption{Removing a small half disk $D^{n+N'+2}$ going inside.}
\label{fig2}
\end{figure}

The converse is easy to prove by reversing the above argument. We remove a small half disk $\frac{1}{2}D^{n+N'+2}$ around a point in $D^{n+q+1}\subset D^{n+N'+2}$, see Figure~\ref{fig2}, such that the resulting space acts as a concordance between $(f,\alpha)$ and $(id,id)$.

\end{proof}

\section{Embeddings modulo immersions as special disked embeddings}\label{sec3}

Let $D^{n+\infty}:= \cup_{N} D^{n+N}$. By a smooth embedding $D^n\hookrightarrow D^{n+\infty}$, we understand $D^n\hookrightarrow D^{n+N}$ for some $N$ large enough.
Set
\sloppy
$$Emb_{\partial}^{fr}(D^n,D^{n+\infty}):=\bigcup_{N}Emb_{\partial}^{fr}(D^n,D^{n+N}),$$
$$Imm_{\partial}^{fr}(D^n,D^{n+\infty}):=\bigcup_{N}Imm_{\partial}^{fr}(D^n,D^{n+N}).$$

Similarly, we define $SDE_n^q$ to be the space of special disked embeddings $(f,\alpha):(S^n,D^{n+1})\hookrightarrow (S^{n+q},D^{n+1+\infty})$ with Type~III framing. 
By construction, $\pi_0SDE_n^q=\overline{FC}{}_n^q$.\\

We claim that the space $SDE_n^q$ has the same set of connected components as the space of embeddings modulo immersions i.e., $\overline{FC}{}_n^q=\pi_{0}\overline{Emb}_{\partial}(D^n,D^{n+q})$. First we prove the following lemma which gives different geometric interpretations of the group $\pi_0\overline{Emb}_{\partial}(D^n,D^{n+q})$.\\

\begin{lemma} \thlabel{lmm} \textit{For} $q\geq 3$,
\begin{multline}\label{eq1}
\pi_{0}\overline{Emb}(S^n,S^{n+q})=\pi_{0}\overline{Emb}_{\partial}(D^n,D^{n+q})=\pi_{0}\overline{Emb}(S^n,\mathbb{R}^{n+q}) \\ =\pi_{0}\overline{Emb}^{fr}(S^n,S^{n+q})=\pi_{0}\overline{Emb}_{\partial}^{fr}(D^n,D^{n+q})=\pi_{0}\overline{Emb}^{fr}(S^n,\mathbb{R}^{n+q}).
\end{multline}

\end{lemma} 

\begin{proof} Let us first prove the \enquote{non-framed} case $$\pi_0\overline{Emb}(S^n,S^{n+q})= \pi_0\overline{Emb}(S^n,\mathbb{R}^{n+q}).$$

Consider the following diagram, where the horizontal lines are fiber sequences:\\
\begin{center}
\begin{tikzcd}
\overline{Emb}(S^n,\mathbb{R}^{n+q}) \arrow[r] \arrow[d] & Emb(S^n,\mathbb{R}^{n+q}) \arrow[r] \arrow[d]  & Imm(S^n,\mathbb{R}^{n+q}) \arrow[d] \\
\overline{Emb}(S^n,S^{n+q})\arrow[r]          & Emb(S^n,S^{n+q}) \arrow[r]                                         & Imm(S^n,S^{n+q})
\end{tikzcd}
\end{center}

\vspace{0.2cm}

We view $\mathbb{R}^{n+q}=S^{n+q}-\infty$, where the \enquote{infinity} point is of codimension $n+q$. Given an embedding (resp. immersion) from $S^n$ to $S^{n+q}$, we can perturb it slightly in a way that it misses the \enquote{infinity} point as $q\geq 3$, so that we get an embedding (resp. immersion) from $S^n$ to $\mathbb{R}^{n+q}$. Therefore, the second (resp. third) vertical map is surjective on the level of $\pi_0$. For injectivity, note that an isotopy (resp. regular homotopy) of an embedding (resp. immersion) of $S^n$ in $S^{n+q}$ is $(n+1)$-dimensional, while the \enquote{infinity} point has codimension $n+q$, so it can still miss the point given $q\geq 3$, and therefore the second (resp. third) vertical map is bijective on $\pi_0$. The same argument holds for $\pi_1$ because $q\geq 3$. Therefore, the second and third vertical maps induce isomorphisms on $\pi_0$ and $\pi_1$ when $q\geq 3$. By five lemma, we get
$$\pi_0\overline{Emb}(S^n,S^{n+q})= \pi_0\overline{Emb}(S^n,\mathbb{R}^{n+q}).$$
It is proved in \cite[Theorem 1.1]{vt3} that $$\pi_0\overline{Emb}(S^n,\mathbb{R}^{n+q})=\pi_0\overline{Emb}_{\partial}(D^n,D^{n+q}).$$
Using the argument from \cite[Proposition 1.2]{song}, we have that the following natural projections are weak equivalences:
$$\overline{Emb}_\partial^{fr}(D^n,D^{n+q})\rightarrow \overline{Emb}_\partial(D^n,D^{n+q}),$$  $$\overline{Emb}^{fr}(S^n,\mathbb{R}^{n+q})\rightarrow \overline{Emb}(S^n,\mathbb{R}^{n+q}).$$
Similarly, one can show that $\overline{Emb}^{fr}(S^n,S^{n+q})\rightarrow \overline{Emb}(S^n,S^{n+q})$ is a weak equivalence. Thus, we get different representations for $\pi_0\overline{Emb}_{\partial}(D^n,D^{n+q})$ as in (\ref{eq1}).

\end{proof}

By Smale-Hirsch theory \cite{hir, sm}, we have $Imm_{\partial}^{fr}(D^n,D^{n+q})\simeq \Omega^n SO(n+q)$, and since we consider the ambient dimension tend to infinity, we get $Imm_{\partial}^{fr}(D^n,D^{n+\infty})\simeq \Omega^n SO$. Note that $Emb_{\partial}^{fr}(D^n,D^{n+N})$ is an open, dense subset of $Imm_{\partial}^{fr}(D^n,D^{n+N})$ of codimension $N-n$. As $N$ gets large, the inclusion $Emb_{\partial}^{fr}(D^n,D^{n+N})\hookrightarrow Imm_{\partial}^{fr}(D^n,D^{n+N})$ becomes highly connected, and we get $Emb_{\partial}^{fr}(D^n,D^{n+\infty})\simeq Imm_{\partial}^{fr}(D^n,D^{n+\infty})\simeq \Omega^n SO$.\\

\begin{lemma}\thlabel{lm} \textit{For} $q\geq 3$, 
$$\pi_{0}\overline{Emb}_{\partial}^{fr}(D^n,D^{n+q})=\pi_{0}\text{hofib}(Emb_{\partial}^{fr}(D^n,D^{n+q})\rightarrow Emb_{\partial}^{fr}(D^n,D^{n+\infty})).$$

\end{lemma}
\begin{proof}
By definition, $\pi_0\overline{Emb}_{\partial}^{fr}(D^n,D^{n+q})$ is equal to $\pi_0$hofib$(Emb_{\partial}^{fr}(D^n,D^{n+q})\rightarrow Imm_{\partial}^{fr}(D^n,D^{n+q})\simeq \Omega^n SO(n+q))$, which is isomorphic to $\pi_0$hofib$(Emb_{\partial}^{fr}(D^n,D^{n+q})\rightarrow \Omega^n SO)$ using the stability of the homotopy groups of $SO$: $$\pi_{i}SO(n+q)=\pi_{i}SO, \text{   \hspace{0.2cm}     if   } i\leq n+q-2.$$ Therefore, for $q\geq 3$, we have that $\pi_{0}\Omega^n SO(n+q)=\pi_{n} SO(n+q)=\pi_{n}SO=\pi_{0} \Omega^n SO$ and similarly $\pi_{1}\Omega^n SO(n+q)=\pi_{1} \Omega^n SO$. Since $ \Omega^n SO \simeq Emb_{\partial}^{fr}(D^n,D^{n+\infty})$, we get the result as a consequence of five lemma.

\end{proof}

Thus, for any element $[(f,\alpha)]$ in $\pi_{0}\overline{Emb}_{\partial}^{fr}(D^n,D^{n+q})$ there corresponds an equivalence class of a pair $(\tilde{f},\tilde{\alpha})$ where $\tilde{f}:D^n\hookrightarrow D^{n+q}$ and $\tilde{\alpha}:[0,1]\rightarrow Emb_{\partial}^{fr}(D^n,D^{n+\infty})$ i.e, $\tilde{\alpha}: D^n\times [0,1]\hookrightarrow D^{n+N}\times [0,1]$ such that $\tilde{\alpha}|_{D^n\times 0}=id$ and $\tilde{\alpha}|_{D^n\times 1}=\tilde{f}$, together with framing.\\

We consider $D^{n+1}\cong D^n\times [0,1]$ obtained by identifying $D_{+}^n$ to $D^n\times \{1\}$ and $D_{-}^n$ to $D^n\times \{0\}\cup S^{n-1}\times [0,1]$, and then smoothening the corners. We similarly identify $D^{n+N+1}\cong D^{n+N}\times [0,1]$, for some large $N$. Therefore, each pair $(\tilde{f},\tilde{\alpha})$ can be thought of as a special disked embedding with Type~III framing i.e., a pair $(id \cup\tilde{f},\tilde{\alpha}):(S^n,D^{n+1})\hookrightarrow (S^{n+q},D^{n+N+1})$ such that $\tilde{\alpha}|_{D_{-}^n}$ is the trivial inclusion $id:D_{-}^n\hookrightarrow D_{-}^{n+q}$ with trivial framing, and $\tilde{\alpha}|_{D_{+}^n}$ is the framed embedding $\tilde{f}:D_{+}^n\hookrightarrow D_{+}^{n+q}$. In other words, one has a natural map 
\begin{equation}\label{mu}
\mu:\text{ hofib}(Emb_{\partial}^{fr}(D^n,D^{n+q})\rightarrow Emb_{\partial}^{fr}(D^n,D^{n+\infty})) \longrightarrow  SDE_n^q.\\
\end{equation}

On the level of $\pi_0$, we obtain $$\mu_*: \pi_0\overline{Emb}_{\partial}^{fr}(D^n,D^{n+q})\rightarrow \overline{FC}{}_n^q,$$ $$[(f,\alpha)]\mapsto [(id\cup\tilde{f},\tilde{\alpha})].$$

\begin{theorem}\thlabel{cor} \textit{For $q\geq 3$, $\mu_*$ is an isomorphism, and therefore}
$$ \pi_0\overline{Emb}_{\partial}^{fr}(D^n,D^{n+q})= \overline{FC}{}_n^q.$$
\end{theorem}

\begin{proof}
To show that $\mu_*$ is bijective, it suffices to show that hofib$(Emb_{\partial}^{fr}(D^n,D^{n+q})\rightarrow Emb_{\partial}^{fr}(D^n,D^{n+\infty}))$ and $SDE_n^q$ are weakly homotopy equivalent, and then Lemma~\ref{lm} concludes the result.\\

Consider the following diagram, where the vertical lines are fiber sequences, and the map $\mu$ is defined above (\ref{mu}).

\begin{center}
\begin{tikzcd}[column sep=-0.5em]
 \Omega Emb_{\partial}^{fr}(D^n,D^{n+\infty}) \arrow[r] \arrow[d] & Emb_{\partial}^{fr}(D^{n+1},D^{n+1+\infty})  \arrow[d] \\
  \text{hofib}(Emb_{\partial}^{fr}(D^n,D^{n+q})\rightarrow Emb_{\partial}^{fr}(D^n,D^{n+\infty})) \arrow[r, "\mu"]  \arrow[d]  & SDE_n^q    \arrow[d] \\
Emb_{\partial}^{fr}(D^n,D^{n+q}) \arrow[r, equal] & Emb_{\partial}^{fr}(D^n,D^{n+q}) \\
\end{tikzcd}
\end{center}
Note that the top map is just restriction on the fibers. Moreover, it is a homotopy equivalence since $Emb_{\partial}^{fr}(D^{n+1},D^{n+1+\infty})\simeq Imm_{\partial}^{fr}(D^{n+1},D^{n+1+\infty})\simeq \Omega^{n+1}SO\simeq \Omega \Omega^{n}SO\simeq \Omega Emb_{\partial}^{fr}(D^n,D^{n+\infty})$. Thus, the map in the middle $\mu$ is also a weak homotopy equivalence. By Lemma~\ref{lm}, we get $ \pi_0\overline{Emb}_{\partial}^{fr}(D^n,D^{n+q})= \overline{FC}{}_n^q$.

\end{proof}

Theorem \ref{mth1} is immediate by combining Lemma \ref{lmm} and Theorems~\ref{mt} and~\ref{cor}.

\section{Geometric interpretation of long exact sequences associated with the triad $(SG;SO,SG_q)$}\label{sec4}

In this section, we prove Theorem \ref{mth2} for the \enquote{non-framed} case i.e., we show the isomorphism between the sequences in (\ref{*}). The proof for the framed case is similar.\\

Recall that $Im_n^q$ is the group of concordance (or equivalently regular homotopy) classes of immersions of $S^n$ in $S^{n+q}$. According to Haefliger \cite[\S 4]{hae1}, any representative in $Im_n^q$ is regular homotopic to a special immersion i.e., an immersion $f:S^n\looparrowright S^{n+q}$ such that $f|_{D_{-}^n}$ is the natural inclusion in $D_{-}^{n+q}$ and $f|_{D_{+}^n}$ is an immersion in $D_{+}^{n+q}$. We can extend this immersion as a disk immersion $\bar{f}:D^{n+1}\looparrowright D^{n+N+1}$ for $N$ large enough. Furthermore, we add framing on $\bar{f}$ by first extending the framing from the base-point $*=e_2$ to $D_{+}^n$ inside $D_{+}^{n+q}$, and then we extend this framing to $D^{n+1}$ inside $D^{n+N+1}$ after taking the suspension. In other words, we add disk structure and Type~I framing in the same way as we did for special embeddings representing elements in $C_n^q$. Thus, any element in $Im_n^q$ can be represented by a special disked immersion $(f,\bar{f}):(S^n,D^{n+1})\looparrowright (S^{n+q},D^{n+N+1})$ with Type~I framing. \\

Haefliger \cite[\S 4.2]{hae} has shown that $Im_n^q$ is isomorphic to the homotopy group $\pi_n(SO,SO_q)$ where his map $\eta:Im_n^q\rightarrow \pi_n(SO,SO_q)$ is defined as follows: given a special disked immersion $(f,\bar{f}):(S^n,D^{n+1})\looparrowright (S^{n+q},D^{n+N+1})$ with Type I framing, one considers the trivialization of the normal bundle induced by the framing of $\bar{f}$. To each $x\in S^n$, one associates the $(N-q)$ frame $e_{n+q+2},..., e_{N+n+1}$ with respect to this trivialization. This $(N-q)$ frame defines a map $h_{f}:S^n\rightarrow V_{N,N-q}$ that represents a homotopy class $[h_f]$ in $\pi_n(V_{N,N-q})=\pi_n(SO,SO_q)$, where $V_{N,N-q}=SO_N/SO_{N-q}$ is the Stiefel manifold.

\begin{remark}\label{rmk}
Note that $h_{f}|_{D_{+}^n}$ is constantly equal to the identity inclusion $\mathbb{R}^{N-q}\subset \mathbb{R}^N$ (viewed as the base-point of $V_{N,N-q}$) because the framing on $D_{+}^n$ is given by suspension and the last $(N-q)$ vectors are $e_{n+q+2},..., e_{N+n+1}$. Hence, the class $[h_{f}]$ depends only on the framing at~$D_{-}^n$. 
\end{remark}

Let us now describe the map $\theta$ appearing in the geometric long exact sequence:
\begin{equation} 
\longrightarrow \overline{FC}{}_n^q\longrightarrow C_n^q\longrightarrow Im_n^q{\overset{\theta}{\longrightarrow}} \overline{FC}{}_{n-1}^q \longrightarrow 
\end{equation}
\vspace{0.01cm}

Note that $Im_n^q=\pi_0Imm_{\partial}(D^n, D^{n+q})=\pi_n V_{n+q,n}=\pi_1Imm_{\partial}(D^{n-1}, D^{n+q-1})$. The natural map $\Omega Imm_{\partial}(D^{n-1},D^{n+q-1})\rightarrow \overline{Emb}_{\partial}(D^{n-1},D^{n+q-1})$ induces a map $Im_n^q=\pi_1Imm_{\partial}(D^{n-1},D^{n+q-1}) \rightarrow \pi_0\overline{Emb}_{\partial}(D^{n-1},D^{n+q-1})=\overline{FC}_{n-1}^q$.  \\

We can also interpret $\theta:Im_n^q\rightarrow \overline{FC}{}_{n-1}^q$ in terms of disked embeddings/immersions as follows: given a special disked immersion $(f,\bar{f}):(S^n,D^{n+1})\looparrowright (S^{n+q}, D^{n+N+1})$ with Type~I framing representing an element in $Im_n^q$, we consider the restriction $f|_{S^{n-1}=D_{-}^n\cap D_{+}^n }=g=id: S^{n-1}\hookrightarrow S^{n+q-1}$, which is the natural inclusion. Moreover, we get the disk immersion $f|_{D_{+}^n}:D_{+}^{n}\looparrowright D_{+}^{n+q}$, which can be immersed inside a bigger disk $D^{n+N}_{+}$ by allowing more dimensions. As a result, we obtain a disk immersion $\bar{g}:=id\circ f|_{D^n_{+}}:D^n_{+}\looparrowright D^{n+N}_{+}$ with the restricted framing from $\bar{f}|_{D^n_{+}}$. Since $N$ is large enough, we can change the framed immersion $\bar{g}$ into a framed embedding $\bar{g}':D^{n}\hookrightarrow D^{n+N}$. The obtained pair $(g,\bar{g}'):(S^{n-1}, D^n)\hookrightarrow (S^{n+q-1},D^{n+N})$ is a disked embedding where the framing on $\bar{g}'|_{S^{n-1}}$ is given by suspension of a framing inside $S^{n+q-1}$ i.e., $(g,\bar{g}')$ is a framed disked embedding. Therefore, given a special disked immersion $(f,\bar{f})$ with Type~I framing, we can assign a framed disked embedding $(g,\bar{g}')$ to it. Thus, we get a well defined map from $Im_n^q$ to $\overline{FC}{}_{n-1}^q$. \\

The commutativity of the following diagram is given by a similar argument as in the proof of Theorem~\ref{cor}.

\begin{center}
\begin{tikzcd}
Im_n^q \arrow[r,"\theta"] & \overline{FC}{}_{n-1}^q  \\
\pi_1Imm(D^{n-1},D^{n+q-1})\arrow[r] \arrow[u, "\simeq"]         & \pi_0\overline{Emb}_{\partial}(D^{n-1},D^{n+q-1}) \arrow[u, "\simeq"]
\end{tikzcd}
\end{center}

\begin{remark}\label{rem}
Note that while defining $\theta$, instead of $(g,\bar{g})=(id,id\circ f|_{D^n_{+}}):(S^{n-1}, D^n)\looparrowright (S^{n+q-1},D^{n+N})$, we can consider the pair $(id,id \circ f|_{D^n_{-}})$ which is same as $(id, id)$ since $f|_{D^n_{-}}=id$ with framing restricted from $\bar{f}$ (such framing may not be a suspension). In $\overline{FC}{}_{n-1}^q$, the representative $(g,\bar{g}')$ corresponding to $(g,\bar{g})=(id, id\circ f|_{D^n_{+}})$ is equivalent to the framed trivial disked embedding $(id,id)=(id,id\circ f|_{D^n_{-}})$ with framing as on $\bar{f}|_{D^n_{-}}$, since $\bar{f}$ acts as a concordance between $id\circ f|_{D^n_{+}}$ and $id\circ f|_{D^n_{-}}$. To be precise, we take a perturbation $\bar{f}'$ of $\bar{f}$ which acts as a concordance. We will use this in the following proof to show commutativity of the third square in (\ref{five}).\\
\end{remark}

\begin{proof}[\textbf{Proof of Theorem \ref{mth2}}] To prove the result, we need to show that each square in the following diagram commutes:
\small
\begin{equation}\label{five}
\begin{tikzcd}[column sep=tiny]
 \pi_{n+1}(SG,SG_q) \arrow[r] & \pi_{n+1}(SG;SO,SG_q) \arrow[r] & \pi_n(SO,SO_q)  \arrow[r] & \pi_n(SG,SG_q) \\
 \overline{FC}{}_n^q \arrow[r]  \arrow[u,"\xi", "\simeq" '] & C_n^q \arrow[r]  \arrow[u, "\psi", "\simeq" '] &  Im_n^q  \arrow[r, "\theta"] \arrow[u, "\eta", "\simeq" ']  & \overline{FC}{}_{n-1}^q \arrow[u, "\xi", "\simeq" '] 
\end{tikzcd}
\end{equation}
\normalsize

For the first square, the map $\overline{FC}{}_n^q \rightarrow C_n^q$ is an inclusion on the level of representatives i.e., a framed disked embedding representing an element in $\overline{FC}{}_n^q$ clearly represents an element in $C_n^q$. Therefore, the commutativity of this square is straightforward from the construction. \\

The commutativity of the second square is given by Haefliger \cite[\S 4.4]{hae} and is easy to see. The map $C_n^q\rightarrow Im_n^q$ is obvious since an embedding is also an immersion. We have seen that the vertical map $\eta$ on a given representative in $Im_n^q$ depends only on the behavior of the representative on $D_{-}^n$, see Remark~\ref{rmk}. Similarly, the top horizontal map $\pi_{n+1}(SG;SO,SG_q) \rightarrow \pi_n(SO,SO_q)$ is defined by restricting the representatives in $\pi_{n+1}(SG;SO,SG_q)$ to the half-disk $D_{-}^n$. \\

We now check the commutativity of the third square. Given an element $\alpha\in Im_n^q$ represented by a special disked immersion $(f,\bar{f}):(S^n,D^{n+1})\looparrowright (S^{n+q}, D^{n+N+1})$ with Type~I framing, by Remark~\ref{rem}, the corresponding element $\theta(\alpha)$ in $\overline{FC}{}_{n-1}^q$ can be represented by the framed trivial disked embedding $(id,id)=(id,id\circ f|_{D^n_{-}}):(S^{n-1},D^n)\hookrightarrow (S^{n-1}\times D^q, D^{n}\times D^N)$, with the framing obtained as a restriction $\bar{f}|_{D^n_{-}}$. Recall that on $S^{n-1}=D_{-}^n\cap D_{+}^n$, the framing is given by suspension of a framing inside $S^{n+q-1}$. We can homotope the obtained framing on $S^{n-1}$ so that it becomes trivial on $D^{n-1}_{-}$. Now, under the vertical map $\xi$, the image of $\theta(\alpha)$ is represented by a map $\phi:D^n\times S^{N-1}\rightarrow S^{N-1}$  with an extension $\bar{\phi}:D^n\times D^N\rightarrow D^N$ defined linearly by $(x,y)\mapsto r(x)(y)$, for some rotation $r$ given by the framing on $\bar{f}|_{D^n_{-}}$. More precisely, $r:D^n\rightarrow SO(N)$ is such that $r|_{\partial D^n=S^{n-1}}$ is a suspension of rotation in $SO(q)$ with $r=id$ on $D^{n-1}_{-}$, by construction. The map $\bar{\phi}$ satisfies the definition of $\xi$ (see subsection \ref{sub3}), since $\bar{\phi}^{-1}(0)=\bar{f}(D^n_{-})=D^n\times 0$, with $\bar{\phi}|_{S^{n-1}}\in SO(q)$ such that $\bar{\phi}_x=id$ for $x\in D^{n-1}_{-}$ and $\bar{\phi}_x$ is the suspension of a map $D^q\rightarrow D^q$ for any $x\in D^{n-1}_{+}$. Moreover, $\bar{\phi}$ also represents an element in $\pi_n(SO,SO_q)$ and is precisely the representative that we get for $\eta(\alpha)$, as $\eta$ also depends only on the non-trivial framing on $D_{-}^n$ (see Remark~\ref{rmk}). Therefore, the square commutes.

\end{proof}

\section{Applications}\label{sec5}
 \subsection{Known computations}
For $C_n^q$, the well-known computations were done by Haefliger \cite{hae3, hae2, hae} in the 1960s and later by Milgram \cite{mil} in the early 1970s. To the best of our knowledge, no computations were done ever since. The Manifold Atlas webpage \cite{max} describes all the known groups $C_n^q$. Haefliger \cite{hae3} has shown that $C_n^q=0$ for $n<2q-3$. Furthermore, he proved that for $q\geq 3$ (see \cite[Corollary~8.14]{hae}),
 \begin{align*}
   C_{2q-3}^{q}= \left\{ \begin{array}{cc} 
            \mathbb{Z}  & \hspace{5mm}  \text{ \textit{q odd}} , \\
             \mathbb{Z}_2 & \hspace{5mm}  \text{ \textit{q even}}. \\
           \end{array} \right.
                                           \end{align*}
For odd $q$, the generator is given by the Haefliger trefoil knot \cite{hae2}. It is an interesting question whether the Haefliger trefoil is a generator for the even case.\\                                       
 
There are only a few computations for $FC_n^q$ in the literature. For example, Haefliger \cite[Theorem~5.17]{hae} has shown that $FC_3^3=\mathbb{Z}\oplus \mathbb{Z}$. Moreover, it is easy to see that for $n<2q-3$, we get $FC_n^q=\pi_n(SO_q)$, see Proposition~\ref{or}. \\ 

The groups $\overline{FC}{}_n^q = \pi_{n+1}(SG,SG_q)$ are related to the homotopy groups of spheres. Some of these groups are known, in particular, $\overline{FC}{}_{2}^{3}=\pi_3(SG,SG_3)=\mathbb{Z}_2$ and $\overline{FC}{}_3^3=\pi_{4}(SG,SG_3)=\mathbb{Z}$, found in \cite[\S 5.16]{hae} and \cite[Proof of Lemma~3.1]{skop}.\\

The rational computations of $\overline {FC}{}_n^q$ are known \cite[Corollary 20]{vt1}, \cite[\S 5.7]{vt2}, and can also be computed directly as follows:
\begin{prop}
\textit{For} $q\geq 3$, \begin{align*}
  \overline {FC}{}_n^q\otimes \mathbb{Q} = \left\{ \begin{array}{cc} 
            \mathbb{Q}  & \hspace{3mm}   n=q-1, \hspace{3mm} q \hspace{2mm}  even , \\
             \mathbb{Q} & \hspace{3mm}  n=2q-3, \hspace{2mm} q \hspace{2mm}odd, \\
               0  & \hspace{2mm} otherwise. \\
                              \end{array} \right.
                                           \end{align*}
\end{prop}
\begin{proof}
Since $\overline {FC}{}_n^q=\pi_{n+1}(SG,SG_q)$, we consider the long exact sequence of the pair $(SG,SG_q)$:
\begin{equation}\label{1}
..\longrightarrow \pi_{n+1}(SG)\longrightarrow \pi_{n+1}(SG,SG_q)\longrightarrow \pi_n(SG_q)\longrightarrow \pi_n(SG)\longrightarrow.. 
\end{equation}

The rational homotopy groups $\pi_n^{\mathbb{Q}}(SG_q)$ can be easily computed by considering the long exact sequence associated with the fibration $\Omega_*^{q-1}S^{q-1}\rightarrow SG_q\rightarrow S^{q-1}$, where $\Omega_*^{q-1}S^{q-1}$ is the component of loops of degree one. The connecting homomorphism $\pi_{n+1}(S^{q-1})\rightarrow \pi_n(\Omega_*^{q-1}S^{q-1})=\pi_{n+q-1}(S^{q-1})$ is given by the Whitehead bracket $[id_{S^{q-1}},-]$. Note that all $\pi_n(SG)$ are torsions being the stable homotopy groups of spheres, hence, $\pi_n^{\mathbb{Q}}(SG)=0$. Using the rational homotopy groups of spheres, we get \begin{align*}
\pi_n^{\mathbb{Q}}(SG_q) = \left\{ \begin{array}{cc} 
          \mathbb{Q}  & \hspace{3mm}   n=q-1, \hspace{3mm} q \hspace{2mm}  even , \\
             \mathbb{Q} & \hspace{3mm}  n=2q-3, \hspace{2mm} q \hspace{2mm}odd, \\
               0  & \hspace{2mm} otherwise. 
                              \end{array} \right.
                                           \end{align*}

Thus, from the long exact sequence (\ref{1}) we get $\pi_{n+1}^{\mathbb{Q}}(SG,SG_q)=\pi_n^{\mathbb{Q}}(SG_q)$ and that concludes the result.
\end{proof}

\subsection{Metastable range}

From \cite[\S 4.4]{hae}, one can deduce the following stability result for the groups $\overline {FC}{}_{n}^q$ and $FC_n^q$:
\begin{prop}\thlabel{or}
\textit{For $n<2q-3$, $\overline{FC}{}_n^q=\pi_{n+1}(SO,SO_q)$ and $FC_n^q=\pi_nSO_q$.}
\end{prop}
\begin{proof}

Consider the long exact sequence associated with the triad $(SG;SO,SG_q)$ given in \cite[\S 4.4]{hae}:
\small
\begin{equation}\label{2}
\rightarrow \pi_{n+1}(SG,SG_q)\rightarrow \pi_{n+1}(SG;SO,SG_q) \rightarrow \pi_n(SO,SO_q)  \rightarrow  \pi_n(SG,SG_q)\rightarrow
\end{equation}
\normalsize
By \cite[Corollary 6.6]{hae}, the groups $\pi_{n+1}(SG;SO,SG_q)=C_n^q=0$ for $n<2q-3$. Therefore, from the above sequence, we get $\pi_{n+1}(SO,SO_q)=\pi_{n+1}(SG,SG_q)=\overline {FC}{}_n^q$ for $n<2q-4$. Moreover, any element of $C_n^q$ is trivial as immersion for $n<2q-1$, see \cite[Corollary~6.10]{hae}. Thus, the homomorphism $\pi_{2q-2}(SG;SO,SG_q)\rightarrow \pi_{2q-3}(SO,SO_q)$ in (\ref{2}) is trivial, and we get $\pi_{2q-3}(SO,SO_q)=\pi_{2q-3}(SG,SG_q)=\overline{FC}{}_{2q-4}^q$.\\

For the second equality, we consider the geometric long exact sequence given by Haefliger \cite[\S 5.9]{hae}:
\begin{equation}\label{3}
\longrightarrow\pi_nSO_q\longrightarrow FC_n^q\longrightarrow C_n^q\longrightarrow \pi_{n-1}SO_q\longrightarrow 
\end{equation}
The result easily follows for $n<2q-4$ since $C_n^q=0$ for $n<2q-3$. When $n=2q-4$, the homomorphism $C_{2q-3}^q\rightarrow \pi_{2q-4}SO_q $ in (\ref{3}) is the composition $C_{2q-3}^q=\pi_{2q-2}(SG;SO,SG_q){\overset{0}\rightarrow} \pi_{2q-3}(SO,SO_q)\rightarrow \pi_{2q-4}(SO_q)$, and therefore is also trivial.

\end{proof}


\begin{lemma}\thlabel{le}
\textit{For $i\leq q-2$, $\pi_i(SG_q)=\pi_{i}(SG)$.}
\end{lemma}
\begin{proof}
When $i\leq q-1$, we have $\pi_i(SG,SG_q)=\pi_{i}(SO,SO_q)=0$, where we get the first equality from Proposition \ref{or}, and the second one using the fact that $\pi_{i}(SO,SO_q)=0$ for $i< q$. Therefore, from the long exact sequence (\ref{1}), we conclude that for $i\leq q-2$, $\pi_i(SG_q)=\pi_{i}(SG)$. 
\end{proof}

\Address

\end{document}